\documentclass[12pt,reqno]{amsart}
\usepackage{latexsym}
\usepackage{amssymb}

\usepackage{epsfig}
\newcommand{\cp}{\mathbb{CP}^1}
\newcommand{\PR}{{\bf Prob}}

\newcommand{\szego}{Szeg\"o\ }

\newcommand{\kahler}{K\"ahler\ }

\newcommand{\PP}{{\mathbb P}}
\newcommand{\R}{{\mathbb R}}
\newcommand{\C}{{\mathbb C}}

\newcommand{\CP}{\C\PP}

\newcommand{\dbar}{\bar\partial}
\newcommand{\ddbar}{\partial\dbar}

\newcommand{\E}{{\mathbf E}}

\newcommand{\half}{{\frac{1}{2}}}

\renewcommand{\phi}{\varphi}

\newcommand{\dcal}{\mathcal{D}}
\newcommand{\ecal}{\mathcal{E}}

\newcommand{\jcal}{\mathcal{J}}
\newcommand{\lcal}{\mathcal{L}}
\newcommand{\pcal}{\mathcal{P}}
\newcommand{\mcal}{\mathcal{M}}

\newcommand{\ocal}{\mathcal{O}}

\newtheorem{theo}{{\sc Theorem}}
\newtheorem{cor}[theo]{{\sc Corollary}}

\newtheorem{lem}[theo]{{\sc Lemma}}
\newtheorem{prop}[theo]{{\sc Proposition}}
\newtheorem{maintheo}{{\sc Theorem}}

\newenvironment{rem}{\medskip\noindent{\it Remark:\/} }{\medskip}

\title[Large deviations for zeros of $P(\phi)_2$ random
polynomials
 ] {Large deviations for zeros of $P(\phi)_2$  random
polynomials}

\author{Renjie Feng and Steve Zelditch}
\address{Department of Mathematics, Northwestern University,
Evanston IL,  60208-2730, USA}

\thanks{Research partially supported by NSF grant  \# DMS-0904252.}

\date{\today}

\begin{document}

\begin{abstract} We extend the results of \cite{ZZ} on LDP's
 (large deviations principles) for the empirical measures $$ Z_s: = \frac{1}{N} \sum_{\zeta: s(\zeta) = 0}
 \delta_{\zeta}, \;\;\; (N: = \# \{\zeta: s(\zeta) = 0)\}$$
 of zeros of Gaussian random polynomials $s$ in one variable to
 $P(\phi)_2$ random polynomials. The speed and rate function are the same
 as in the associated Gaussian case. It follows that the expected
 distribution of zeros in the $P(\phi)_2$ ensembles tends to the
 same equilibrium measure as in the Gaussian case.

 \end{abstract}

 \maketitle

The purpose of this note is to extend  the LDP (large deviation
principle) of \cite{ZZ} for the empirical measure
 \begin{equation}\label{ZN} Z_s: = d\mu_{\zeta}: =  \frac{1}{N}
\sum_{\zeta: s(\zeta) = 0}
 \delta_{\zeta}, \;\;\; N: = \# \{\zeta: s(\zeta) = 0\} \end{equation}
  of zeros of  Gaussian
random holomorphic polynomials $s$ of degree $N$ in one variable
to certain non-Gaussian measures which we call  $P(\phi)_2$ random
 polynomials. These are finite dimensional analogues of (or approximations to) the
 ensembles of quantum field theory,  where the probability measure on the space
 of functions (or distributions) has the form $ e^{- S(f)} df$, with
\begin{equation}\label{Sa}  S(f) = \int (|\nabla f|^2 +  |f|^2 + Q(|f|^2)d
\nu,
\end{equation} where $Q$ is a semi-bounded polynomial. A more precise definition is given below;
we  refer to  \cite{Si} for background on $P(\phi)_2$ theories.
Our main results are that the empirical measures of zeros for such
$P(\phi)_2$ random polynomials satisfies an LDP with precisely the
same speed and rate functional as in the Gaussian case in
\cite{ZZ} where $Q = 0$. In fact, our proof is to reduce the LDP
to that  case.  As a corollary, the expected distribution
$\frac{1}{N} \E_N Z_s$ of zeros in the $P(\phi)_2$ case tends to
the same  weighted equilibrium measure as in the Gaussian case. In
the Gaussian case, the proof of the last statement is derived from
the asymptotics of the two point function (see  \cite{SZ1,SZ2,B});
in the $P(\phi)_2$ case, the large deviations proof  is the first
and only one we know.

To state the result precisely, we need some notation and terminology.
 By a random polynomial, one means a probability
measure $\gamma_N$ on the vector space $\pcal_N$ of polynomials
$p(z) = \sum_{j = 0}^N a_j z^j$ of degree $N$. As in \cite{ZZ}, we
identify polynomials $p(z)$ on $\C$ with holomorphic sections $s \in
H^0(\CP^1, \ocal(N))$, where  $\ocal(N)$ is the $N$th power of the hyperplane section
line bundle $\ocal(1)$; strictly speaking, in the local coordinate, $s = p e^N$ where $p$
is the polynomial of degree $N$ and $e^N$ is a frame for $\ocal(N)$. The
geometric language is useful for compactifying the problem to
$\CP^1$, and we refer to \cite{SZ1,ZZ} for background.

 In \cite{ZZ}, the authors chose
$\gamma_N$ to be a Gaussian measure, $$\gamma_N = e^{-
||s||_{(h^N, \nu)}^2} ds,$$ determined by an inner product on
$\pcal_N$,
\begin{equation} \label{INNERPRODUCT} ||s||^2_{(h^N, \nu)} :=
\int_{\CP^1} |s(z)|^2_{h^N} d\nu(z).
\end{equation}
Here, $\nu$ is an auxiliary probability measure and  $h $ is a smooth Hermitian metric on $\ocal(1)$ and $h^N$
is the induced metric on the powers $\ocal(N)$. In the local frame $e$,
$h$ takes the classical form of a weight $h = e^{- \phi}$; the
assumption is that it extends smoothly to $\ocal(1) \to \CP^1$.
Thus in the local coordinate, we rewrite
\begin{equation}  ||s||^2_{(h^N, \nu)} =  \int_{\C} |p(z)|^2 e^{- N
\phi(z)}  d\nu(z).
\end{equation}

In this article, we study  the probability measures
\begin{equation} \label{GL} \gamma_N = e^{- S(s)} ds \;\;\; \mbox{on}\;\; \pcal_N, \end{equation}
where  $ds$ denotes Lebesgue measure and the action $S$ has the
form,
\begin{equation}\label{S}  S(s) = \int_{\CP^1} |\nabla s(z)|_{h^N \otimes
g}^2 d\nu
 +  \int_{\CP^1} P(|s|_{h^N}^2) d \nu,
\end{equation}
where \begin{equation}\label{mnS} P(x) = \sum_{j = 1}^k c_j x^j, \;\; \mbox{with}\; c_k = 1
\end{equation} is a semi-bounded polynomial. Here,  $\nabla: C^{\infty}(\CP^1,
\ocal(N)) \to C^{\infty}(\CP^1, \ocal(N) \otimes T^*)$ is a smooth
connection on the line bundle $\ocal(N) \to \CP^1$, and  $g$ is a
smooth Riemannian metric on $\mathbb{CP}^1$. We recall that
connections are the first order derivatives which are well-defined
on sections of line bundles. We will take $\nabla$ to be the Chern
connection of a smooth connection  $h$ on $\ocal(1)$ and its
extension to the tensor powers $\ocal(N)$ (which strictly speaking
should be denoted by $\nabla_N$). Note that the more elementary
holomorphic derivative $\partial p(z) = p'(z)$ defines a
meromorphic connection on $\ocal(N)$ with a pole at infinity,
rather than a smooth connection. We refer to \S \ref{KINETIC} and
\cite{GH,ZZ} for further background.

The integral $\int_{\CP^1} |\nabla s(z)|_{h^N \otimes g}^2 d\nu$
is expressed in (\ref{kic}) in local coordinates.  We often denote
the first integral in  $S(s)$ as $\|\nabla s\|_{(h^N\otimes g,
\nu)}^2$ and the second as $\int P(|s|_{h^N}^2)$. In $P(\phi)_2$
Euclidean quantum field theory,
 $S(s)$ is known as the action, $\|\nabla s\|^2$
is known as the kinetic energy term,  $P$ is the potential,  and
$\lcal(s) = |\nabla s|^2 + P(|s|^2)$ is  the Lagrangian (see e.g.
\cite{GJ,Si}). The Gaussian case is the `non-interacting' or free
field theory with quadratic Lagrangian $\lcal_0 = |\nabla s|^2 +
m|s|^2$; while in the general $P(\phi)_2$ case, the non-quadratic
part of $P$ is known as the interaction term. The Gaussian case
was studied in \cite{ZZ} without the (also Gaussian) kinetic term.

The large deviations result for empirical measures of zeros
concerns a sequence $\{\PR_N\}$ of probability measures on the
space $\mcal(\CP^1)$ of probability measures on $\CP^1$. Roughly,
$\PR_N(B)$ is the probability that the empirical measure of zeros
of a random $p \in \pcal_N$ lies in the set $B$. To be precise, we
recall some of the definitions from \cite{ZZ}.
 The zero set $\{\zeta_1, \dots,
\zeta_N\}$ of a polynomial of degree $N$ is a point of the $N$th
configuration space,
 \begin{equation} \label{CONFIG} (\CP^1)^{(N)} = Sym^{N} \CP^1 :=  \underbrace{\CP^1\times\cdots\times \CP^1}_N /S_{N}. \end{equation}
   Here, $S_N$ is the symmetric
  group on $N$ letters. We  push forward the measure $\gamma_N$ on
  $\pcal_N$ under the  `zeros'  map
  \begin{equation} \label{ZEROSMAP} \dcal: \pcal_N \to (\CP^1)^{(N)}, \;\;\;\dcal(s)= \{\zeta_1,  \dots,  \zeta_N\}, \end{equation}
  where $\{\zeta_1, \dots, \zeta_N\}$ is the zero set of $s$,
  to obtain a measure
  \begin{equation} \label{JPCDEF}\vec K^N(\zeta_1, \dots, \zeta_N) : =  \dcal_* d \gamma_N \end{equation}
  on $ (\CP^1)^{(N)}$, known as the joint probability
 current (or distribution), which we abbreviate by JPC.
We then embed the configuration
  spaces into $\mcal(\CP^1)$ (the space of probability
  measures on $\CP^1$) under the map,
 \begin{equation}\label{DELTADEF} \mu : (\CP^1)^{(N)} \to \mcal(\CP^1), \;\;\;  d\mu_{\zeta}: =
  \frac{1}{N} \sum_{j = 1}^{N}
 \delta_{\zeta_j}. \end{equation} The measure $d\mu_{\zeta}$ is known as the empirical measure
 of zeros of $p$.  We then  push
forward the joint probability current
 to obtain a probability  measure
\begin{equation} \label{LDPNa} \PR_N =  \mu_* \dcal_* \gamma^N \end{equation}  on
 $\mcal(\CP^1)$. The sequence $\{\PR_N\}$  is said to satisfy a large deviations
 principle with speed $N^2$ and rate functional (or rate function) $I$  if (roughly speaking)
 for any Borel subset $E
\subset \mcal(X)$,
$$\frac{1}{N^2} \log \PR_N \{\sigma \in \mcal: \sigma \in E\}
\to - \inf_{\sigma \in E} I(\sigma). $$ To be precise, the
condition is that
\begin{equation}
    \label{eq-ref1}
    - I(\sigma):= \limsup_{\delta \to 0} \limsup_{N \to \infty}
\frac{1}{N^2} \log {\bf Prob} _N(B(\sigma, \delta)) =
\liminf_{\delta \to 0} \liminf_{N \to \infty} \frac{1}{N^2} \log
{\bf Prob}_N(B(\sigma, \delta)),
\end{equation}
for balls in the natural (Wasserstein) metric (see Theorem 4.1.11
of \cite{DZ}).

\subsection{Statement of results}

Our first results give an LDP for  slightly simpler $P(\phi)_2$
ensembles where the action does
 not contain the kinetic term, i.e., we choose the probability measure to be $\gamma_N=e^{-S(s)}ds$
  where $S(s)=\int P(|s|_{h^N}^2)$. In \S \ref{KINETIC} we add the
  kinetic term.

To obtain a large deviations result,  we need to impose some
conditions on the probability measure $\nu$ that is used to
defined the integration measure on $\CP^1$ in the inner product
(\ref{INNERPRODUCT}) and the $P(\phi)_2$ measures (\ref{GL}). In
the pure potential case in \S \ref{without}, it  must satisfy the
mild conditions of \cite{ZZ}:  (i) the Bernstein-Markov condition,
and (ii) that the support $K$ of $\nu$ must be `regular' in the
sense that it is non-thin at all of its points. We call such
measures {\it admissible}. We refer to \cite{B,ZZ} for background
on Bernstein-Markov measures and regularity. When we include the
kinetic term, we must assume more about $\nu$ (see below).

  If $\gamma_N$ is defined by an \textit{admissible} measure $\nu$,
  then we prove that the speed and the rate function are the same as in the associated Gaussian case \cite{ZZ} where $P(x)=x$.

\begin{maintheo}\label{POTENTIAL}  Let $h = e^{- \phi}$ be a smooth Hermitian metric
 on $\ocal(1) \to \CP^1$ and let
$\nu \in  \mcal(\CP^1)$  be an admissible measure. Let
$P(|s|^2_{h^N})$ be a semi-bounded  polynomial defined by
(\ref{mnS}), and let $\gamma_N$ be the probability measure defined
by the action $S(s)=\int_{\mathbb{CP}^1}P(|s|^2_{h^N})d\nu$
without the kinectic term.
  Then the sequence of probability measures
$\{ \PR_N\}$ on $\mcal(\CP^1)$ defined by (\ref{LDPNa}) satisfies
a large deviations principle with speed $N^2$ and rate functional
\begin{equation} \label{IGREEN}  I^{h, K} (\mu) =
- \frac{1}{2} \ecal_{h}(\mu) + \sup_K U^{\mu}_{h} + E(h) .
\end{equation}
This rate functional is lower semi-continuous, proper and convex,
and its  unique minimizer $\nu_{h, K} \in \mcal(\CP^1)$ is   the
Green's equilibrium measure of $K$ with respect to $h$.
\end{maintheo}

Here, $\ecal_h(\mu) = \int_{\CP^1 \times \CP^1} G_h(z,w) d\mu(z)
d\mu(w)$ is the Green's energy, where $G_h(z,w)$ is the Green's
function with respect to $h$ (see \cite{ZZ} (6)). Also,
$U_h^{\mu}(\mu) = \int_{\CP^1} G_h(z,w) d\mu(w)$ is the Green's
potential of $\mu$.

Things become more complicated when the action includes the
kinetic term. We  could choose independently the integration
measures in the kinetic and potential terms, but for  the sake of
simplicity we only use  the same measure $\nu$ for both terms.  We
then impose an extra condition  on $\nu$ (and $\nabla$), namely
that $\nabla$  satisfies a  weighted $L^2$ Bernstein inequality,
\begin{equation} \label{BERN} \|\nabla s\|^2_{(h^N \otimes
g,\nu)}\leq CN^k \|s\|^2_{(h^N, \nu)}
\end{equation}  on all $H^0(\CP^1, \ocal(N))$, for some  $k,
C(h,g,\nu)> 0.$ When $\nu$ is admissible and such bounds hold, we
say that $\nu$ (or $(h, \nu, \nabla)$) is {\it kinetic
admissible}.
  In Lemma \ref{volume}, we show that if
$h=e^{-\phi}$ is a Hermitian metric on $\ocal(1)$ with positive
curvature form $\omega_h$ and $g$ is any fixed Riemannian metric,
then $\nu = \omega_h$ is kinetic admissible, and in fact
\begin{equation}\label{crucialt}\|\nabla s\|^2_{(h^N \otimes
g,\nu)}\leq CN^2\|s\|^2_{(h^N, \nu)}. \end{equation}

We then extend Theorem \ref{POTENTIAL} to the full $P(\phi)_2$
case. Perhaps surprisingly, when $(h,\nu, \nabla)$ is kinetic
admissible,  the kinetic term becomes a `lower order term' if
$P(x)$ contains non-quadratic terms.

\begin{maintheo}\label{POTENTIALKINETIC}  Let
 $(h,\nu, \nabla)$  be  kinetic  admissible in the sense that (\ref{BERN}) holds.  Let
$P(|s|^2_{h^N})$ be a semi-bounded  polynomial as above,  and let
$\gamma_N$ be the associated $P(\phi)_2$ measure defined by the
action (\ref{S}).
  Then the sequence of probability measures
$\{ \PR_N\}$ on $\mcal(\CP^1)$ defined by (\ref{LDPNa}) satisfies
a large deviations principle with speed $N^2$ and the same  rate
functional
  $I^{h, K} (\mu)$ as in Theorem \ref{POTENTIAL}.

\end{maintheo}

The proofs of Theorems \ref{POTENTIAL}-\ref{POTENTIALKINETIC} are
to relate the LDP for the $P(\phi)_2$ ensemble to the LDP for the
associated (quadratic) Gaussian ensemble without kinetic term
studied in \cite{ZZ}.  To avoid duplication, we refer the reader
to the earlier article for steps in the proof which carry over to
$P(\phi)_2$ measures with no essential change. There are two new
steps that are not in \cite{ZZ}. The first new step (Propositions
\ref{FSVOLZETA2intro} and \ref{FSVOLZETA2introb}) is the
calculation of the JPC (joint probability current, or
distribution) of zeros in the $P(\phi)_2$ ensembles. The main
observation underlying this note is that the calculation of the
JPC in the Gaussian ensemble in \cite{ZZ} extends easily to the
$P(\phi)_2$ case.  The second new step  (loc. cit.) is the
reduction of the proof of the LDP to that of \cite{ZZ} by bounding
the approximate rate function in the $P(\phi)_2$ case above and
below by that in the Gaussian case.

As a direct consequence of Theorems
\ref{POTENTIAL}-\ref{POTENTIALKINETIC}  we obtain,
\begin{cor}\label{EQDIST} With all assumptions in Theorems \ref{POTENTIAL}-\ref{POTENTIALKINETIC} , let $\E_N (Z_s)$ be the expected value of
the empirical measure with respect to $\gamma_N$. Then, $\E_N(Z_s)
\to \nu_{h, K}$ which is the equilibrium measure determined by $h$
and $K$. \end{cor}

Indeed, the limit measure $\lim_{N \to \infty} \E_N(Z_s)$ must be
the unique minimizer of the rate functional.  Convergence of the
expected distribution of zeros to the equilibrium measure was
first proved for Gaussian random polynomials with `subharmonic
weights' in \cite{SZ1} and for flat weights and real analytic $K$
in \cite{SZ2}. In \cite{B}, the flat result was generalized to
admissible measures. Corollary \ref{EQDIST} is the first result to
our knowledge for probability measures of the form (\ref{GL}). In
fact, we are not aware of prior results on these finite
dimensional approximations to $P(\phi)_2$  quantum field theories.
The results may have an independent interest in illustrating a
novel kind of high frequency cutoff for such theories (in a
holomorphic sector).

In conclusion, we thank O. Zeitouni for discussions and
correspondence on this note.

\subsection{An example:  Kac-Hammersley}

As an illustration of the methods and results, we consider a
$P(\phi)_2$ generalization of the  Kac-Hammersley ensemble. The
classical Kac-Hammersley ensemble is the Gaussian random
polynomial
$$s(z)=\sum_{j=0}^N a_j z^j, \,\,\,\, z\in \mathbb{C}$$ where the coefficients $a_j$ are
independent complex Gaussian random variable of mean 0 and
variance 1. In this case,  $\E(Z_s)\rightarrow \delta_{S^1}$ as
the week limit.

In the Gaussian case,   $d\nu = \delta_{S^1}$ (the invariant
probability measure on the unit circle), the weight $e^{-\phi} =
1$ and $g$ is the flat metric. Hence the inner product
(\ref{INNERPRODUCT}) reads
$$\|s\|_{\delta_{S^1}}^2=\frac{1}{2\pi}\int_{0}^{2\pi}|s(e^{i\theta})|^2d\theta$$
where $s$ is a polynomial of degree $N$.

We now use the same metrics and measures, together with any
semi-bounded polynomial $P(|s|^2)$,  to define the kinetic
$P(\phi)_2$ Kac-Hammersley ensemble. We  note that $\delta_{S^1}$
is \emph{admissible} \cite{ZZ}. Second, inequality
(\ref{crucialt}) holds for any polynomials: In the setting of
Kac-Hammersley, the connection $\nabla$ is equal to
$d=\partial+\bar \partial$, thus
$$\nabla s=(\sum _{j=1}^N ja_j z^{j-1})dz$$ thus  $$\|\nabla
s\|_{\delta_{S^1}}^2=\sum _{j=1}^N j^2|a_j|^2 \leq N^2 \sum
_{j=0}^N |a_j|^2=N^2\| s\|_{\delta_{S^1}}^2$$ Hence,  Theorems
\ref{POTENTIAL} - \ref{POTENTIALKINETIC} hold in this case and we
have
\begin{cor} In the setting of  Kac-Hammersley, let  Let $\E_N(Z_s)$ be the expected
value of the empirical measure with respect to $\gamma_N$ defined by $P(\phi)_2$ action (\ref{S}) with the kinetic term. Then, $\E_N(Z_s)\rightarrow \delta_{S^1}$.
\end{cor}

\section{Proof of the Theorem \ref{POTENTIAL}}\label{without}

In this section, we drop  the kinetic  term $\|\nabla s\|_{(h^N
\otimes g,\nu)}^2 $ and only consider actions of the form $\int
P(|s|^2_{h^N}) d\nu$.  We assume that $c_{k} > 0$ and with no
essential loss of generality we put $c_{k} = 1$ (the coefficient
could be re-scaled in the calculation). The following calculation
generalizes Proposition 3 of \cite{ZZ}.

\begin{prop} \label{FSVOLZETA2intro} Let $(\pcal_N, \gamma_N)$ be
the $P(\phi)_2$ ensemble with $S(s) = \int_{\CP^1} P(|s|_{h^N})
d\nu$, where $d\nu$ is an admissible measure. Denote by $k$ the
maximal non-zero power occurring in $P$ (\ref{mnS}).  Let $\vec
K^N$ be the joint probability current (\ref{JPCDEF}). Then,
\begin{eqnarray}
    \label{eq-030209ba}
    \vec K^N(\zeta_1, \dots, \zeta_N) & = & \frac{(\Gamma_N(\zeta_1, \dots \zeta_N))}{Z_N(h)}
\frac{|\Delta(\zeta_1, \dots, \zeta_N)|^2  d^2 \zeta_1 \cdots d^2
\zeta_N}{\left(\int_{\CP^1} \prod_{j = 1}^N |(z - \zeta_j)|^{2k}
e^{-k N \phi(z)}  d\nu(z) \right)^{\frac{N+1}{k}}}
 \\
\label{eq-030209d} & = & \frac{(\Gamma_N(\zeta_1, \dots
\zeta_N))}{\hat{Z}_N(h)} \frac{\exp \left( \sum_{i < j}
G_{h}(\zeta_i, \zeta_j) \right) \prod_{j = 1}^N e^{- 2 N
\phi(\zeta_j)} d^2 \zeta_j  }{\left(\int_{\CP^1} e^{k N
\int_{\CP^1} G_{h}(z,w) d\mu_{\zeta}} d\nu(z)
\right)^{\frac{N+1}{k}}}.
\end{eqnarray}
where $$\sup_{\{\zeta_1, \dots, \zeta_N\} \in (\CP^1)^{(N)}}
\frac{1}{N^2} \log \Gamma_N (\zeta_1, \dots, \zeta_N) \to 0$$ and
where $ Z_N(h)$, resp. $\hat{Z}_N(h)$, is the normalizing constant
in Proposition 3 of \cite{ZZ}.
\end{prop}

We note that (\ref{eq-030209ba}) (resp. (\ref{eq-030209d})) is
almost the same as (23) (resp. (24))  in Proposition 3 of
\cite{ZZ} except that we raise $||s||_{h^N}$ to the power $k$
instead of the power $k = 2$. It is shown in \cite{ZZ} that
$\frac{1}{\hat{Z}_N} = e^{-\frac{1}{2}N(N-1)+N(N+1))E(h)}. $ The
existence of such an explicit JPC in the general $P(\phi)_2$ case
is the reason why it is possible to prove Theorem \ref{POTENTIAL}.

\begin{proof}

We coordinatize $\pcal_N$ using the  basis $z^j$ and put
$$s = a_0 \prod_{j = 1}^N (z - \zeta_j) = \sum_{j = 0}^N a_{N -j} z^j. $$
Any smooth probability measure on $\pcal_N$ thus has a density
$\dcal(a_0, \dots, a_N) \prod_{j = 0}^N d^2 a_j, $ where $d^2 a =
da \wedge d \bar{a}$ is Lebesgue measure.

As in \cite{ZZ}, the first step is to push this measure  forward
under the natural projection from $\pcal_N$ to the projective
space $\PP \pcal_N$ of polynomials, whose points consists of lines
$\C s$ of polynomials. This is natural since $Z_s$ is the same for
all multiples of $s$. Monic polynomials with $a_0 = 1$ form an
affine space of $\PP \pcal_N$. As affine coordinates on $\PP
\pcal_N$ we use $[1:b_1:\cdots :b_N]$ with $b_j=a_j/a_0$.

We then change variables from the affine coordinates $b_j$ to the
zeros coordinates $\zeta_k$. Since $a_{N - j} = e_{N - j}(\zeta_1,
\dots, \zeta_N)$ (the $(N - j)$th elementary symmetric
polynomial), the pushed forward  probability measure on $\PP
\pcal_N$ then  has the form \begin{equation} \label{JPCa}  \vec
K^N(\zeta_1, \dots, \zeta_N) =  \left(\int \dcal(a_0; \zeta_1,
\dots, \zeta_N) |a_0|^{2 N} d^2a_0 \right) \times |\Delta(\zeta_1,
\dots, \zeta_N)|^2 d^2\zeta_1 \cdots d^2\zeta_N, \end{equation}
where $\dcal(a_0; \zeta_1, \dots, \zeta_N)$ is the density of the
JPC in the coordinates $(a_0, \dots, a_N)$ followed by the change
of coordinates, and $\Delta(\zeta_1, \dots, \zeta_N) = \prod_{i <
j} (\zeta_i - \zeta_j)$ is the Vandermonde determinant. We refer
to \cite{ZZ} (proof of Proposition 3) for further details.

For the $P(\phi)_2$ measures (\ref{GL}) without a kinetic term,
\begin{equation} \label{JPCb}  \dcal(a_0; \zeta_1, \dots, \zeta_N) =
e^{- \int_{\CP^1}  P (|a_0|^2 | \prod_{j = 1}^N (z -
\zeta_j)|^2_{h^N}) d\nu(z)}. \end{equation} Put
\begin{equation} \label{alphaj} \alpha_i(\zeta_1 \dots, \zeta_N) : = \alpha_i=\int_{\cp}|\prod_{j=1}^N(z-\zeta_j)|^{2i}_{h^N}d\nu(z).
\end{equation}  Then
\begin{equation} \label{DCAL}  \dcal(a_0; \zeta_1, \dots, \zeta_N) =
e^{-(\alpha_k|a_0|^{2k}+\alpha_{k-1}c_{k - 1} |a_0|^{2k-2}+\cdots
+ \alpha_1 c_1  |a_0|^2)},  \end{equation} and the pushed-forward
density is
\begin{equation} \label{DCAL} \begin{array}{l}  \int \dcal(a_0; \zeta_1,
\dots, \zeta_N) |a_0|^{2 N} d^2a_0  \\ \\=
\int_{\mathbb{C}}e^{-(\alpha_k|a_0|^{2k}+\alpha_{k-1}c_{k - 1}
|a_0|^{2k-2}+\cdots + \alpha_1 c_1  |a_0|^2)}|a_0|^{2N}da_0 \wedge
d \bar a_0. \end{array} \end{equation}  We change variables to
$\rho = |a_0|^2 \to \alpha_k^{-\frac{1}{k}} \rho$ to get
\begin{equation} \label{alphak} \int_0^{\infty}e^{-(\alpha_k\rho^{k}+\alpha_{k-1} c_{k-1}\rho^{k-1}+\cdots
+ \alpha_1 c_1 \rho)} \rho^{N}d\rho=(\alpha_k)^{\frac{N+1}{k}}
\Gamma_N,
\end{equation} where
 \begin{equation} \label{GAMMADEF} \Gamma_N(\zeta_1, \dots,
\zeta_N): = \int_0^{\infty}e^{-(\rho^{k}+\beta_{k-1} c_{k-1}
\rho^{k-1}+\cdots + \beta_1 c_1 \rho)}\rho^{N}d\rho.
\end{equation} with
$\beta_i=\frac{\alpha_i}{\alpha_k^{\frac{i}{k}}}. $ We observe
that
\begin{equation} \label{alphajk} (\alpha_k)^{\frac{N+1}{k}} = \left(\int_{\cp}|\prod_{j=1}^N(z- \zeta_j)|^{2k}_{h^N}d\nu(z)
\right)^{\frac{N + 1}{k}},
\end{equation}
so that (\ref{alphak}) implies the identity (\ref{eq-030209ba}).
The identity (\ref{eq-030209d}) is derived from
(\ref{eq-030209ba}) exactly as in Proposition 3 of \cite{ZZ}, so
we refer there for the details.

To complete the proof of the Proposition, we prove the key

\begin{lem} \label{BOUNDS} We have,

$$\sup_{\{\zeta_1, \dots, \zeta_N\} \in (\CP^1)^{(N)}} \frac{1}{N^2}\log  \Gamma_N(\zeta_1, \dots, \zeta_N) \rightarrow 0$$ \end{lem}

\begin{proof}

By the H\"older inequality with exponent $\frac{k}{i}$,  $\beta_i
\leq (\int_{\cp}d\nu)^{1-\frac{i}{k}}= 1$, hence $\beta_i$ is
bounded independent of $N$ for any polynomial $s$ or roots
$\{\zeta_1, \dots, \zeta_N\}$.

 We first note that
 $$\begin{array}{lll} \rho^{k}+\beta_{k-1} c_{k-1} \rho^{k-1}+\cdots + \beta_1 c_1 \rho &\geq & \rho^{k} - | c_{k-1}| \rho^{k-1} - \dots - |c_1|  \rho
 \\ && \\
  & \geq & \frac{1}{2} \rho^k, \;\;\;\mbox{for}\;\; \rho \geq \rho_k : =  \rho_k(c_1, \dots, c_{k-1}), \end{array} $$
  where $\frac{|c_1|}{\rho} + \cdots + \frac{|c_1|}{\rho^{k-1}}
  \leq \half$ for $\rho \geq \rho_k$.
It follows that
$$\begin{array}{lll} \Gamma_N(\zeta_1, \dots, \zeta_N)  & \leq & \int_0^{\rho_k}  e^{-(\rho^{k}+\beta_{k-1} c_{k-1}
\rho^{k-1}+\cdots + \beta_1 c_1 \rho)}\rho^{N}d\rho +
\int_{\rho_k}^{\infty} e^{- \half \rho^k} \rho^N d \rho \\ && \\
& \leq &   \int_0^{\rho_k}  e^{-(\rho^{k} -
|c_{k-1}|\rho^{k-1}-\cdots - | c_1| \rho)} \rho^{N}d\rho+
\int_{0}^{\infty} e^{- \half \rho^k} \rho^N d \rho.
\end{array}
$$
But
$$
\int_0^{\infty}e^{- \half \rho^{k}}\rho^{N}d\rho
=N^{\frac{N+1}{k}} \int_0^\infty e^{N(\log \rho- \half
\rho^k)}d\rho \sim N^{\frac{N+1}{k}} e^{N(\frac{1}{k}\log
\frac{1}{k}-\frac{1}{k})}\frac{1}{\sqrt{N}}.$$ Also,
$$\int_0^{\rho_k} e^{-(\rho^{k} -  |c_{k-1}|\rho^{k-1}- \cdots
- | c_1 | \rho)} \rho^{N}  d \rho \leq (\rho_k)^N C_k,$$ where
$C_k$ is a constant independent of $N$ and $\{\zeta_1, \dots,
 \zeta_N\}$.
 Hence,
 $$\Gamma_N \leq (\rho_k)^N C_k + N^{\frac{N+1}{k}} e^{N(\frac{1}{k}\log
\frac{1}{k}-\frac{1}{k})}\frac{1}{\sqrt{N}}.$$

To obtain a lower bound, we write
$$\int_0^{\infty}e^{-(\rho^{k}+\beta_{k-1} c_{k-1} \rho^{k-1}+\cdots +
\beta_1 c_1 \rho)}\rho^{N}d\rho= \int_0^{1} +\int_1^{\infty}$$ For
$\rho \in [0, 1]$ we have,
$$\rho^{k}+\beta_{k-1} |c_{k-1}| \rho^{k-1}+\cdots + \beta_1 |c_1| \rho \leq k C, \;\;\; C = \max\{|c_j|\}_{j = 1}^k$$
since each $\beta_i$ is bounded by $1$, thus
$$\int_0^{1}e^{-(\rho^{k}+\beta_{k-1} c_{k-1} \rho^{k-1}+\cdots + \beta_1 c_1 \rho)}\rho^{N}d\rho \geq \int_0^{1} e^{- C k}\rho^{N}d\rho \geq  e^{- Ck}\frac{1}{N+1}$$
For $\rho \geq 1$ we have,
$$\rho^{k}+\beta_{k-1} c_{k-1} \rho^{k-1}+\cdots + \beta_1 c_1 \rho \leq
k C  \rho^{k},$$ hence
$$\begin{array}{l} \int_1^{\infty}e^{-(\rho^{k}+\beta_{k-1} c_{k-1} \rho^{k-1}+\cdots +
\beta_1 c_1 \rho)}\rho^{N}d\rho \geq \int_1^{\infty}e^{- C
k\rho^k}\rho^Nd\rho\\ \\ = (Ck)^{- (N + 1)/k}  \int_1^{\infty}e^{-
\rho^k} \rho^N d\rho  \geq (Ck)^{- (N + 1)/k}.
\end{array} $$ Putting
together the two bounds, we get
$$(Ck)^{- (N + 1)/k} +e^{-k}\frac{1}{N+1} \leq \Gamma_N
 \leq  (\rho_k)^N C_k + N^{\frac{N+1}{k}} e^{N(\frac{1}{k}\log
\frac{1}{k}-\frac{1}{k})}\frac{1}{\sqrt{N}}.$$

This completes the proof of Lemma \ref{BOUNDS}, and hence of the
Proposition.
\end{proof}

\begin{rem} In retrospect, what we proved is that
\begin{equation} \label{DCALab} \frac{1}{N^2} \log \int \dcal(a_0; \zeta_1,
\dots, \zeta_N) |a_0|^{2 N} d^2a_0  \sim \frac{1}{N^2} \log
\int_{0}^{\infty} e^{-\alpha_k \rho^{k}} \rho^{N} d \rho
\end{equation}

We could obtain the limit by a slight generalization of the saddle
point method,
\begin{equation} \begin{array}{lll}  \frac{1}{N^2} \log
\int_{0}^{\infty} e^{-\alpha_k \rho^{k}} \rho^{N} d \rho & \sim &
-\frac{1}{N^2} \inf_{\rho \in \R_+} (\alpha_k \rho^{k} - N \log
\rho) \\ && \\ & \sim &  -\frac{1}{k N}  \log \alpha_N = -
\frac{1}{k N} \log
\int_{\cp}|\prod_{j=1}^N(z-\zeta_j)|^{2k}_{h^N}d\nu,
\end{array} \end{equation} since the minimum occurs at $\rho_N =
(\frac{N}{k})^{\frac{1}{k}} \alpha^{- \frac{1}{k}}. $ This is the
same answer we are about to get by the more rigorous argument in
\cite{ZZ}.

\end{rem}

\subsection{\label{COMPLETE} Completion of the proof of Theorem \ref{POTENTIAL} without kinetic term }

We now modify the calculations of \cite{ZZ}, Section 4.7, of the
approximate rate function $I_N$. As in that section,  we define
$$\mathcal{E}_N^h(\mu_{\zeta})=\int_{\cp \times \cp \backslash
\Delta}G_h(z,w)d\mu_{\zeta}(z)d\mu_{\zeta} (w),$$ where $\Delta
\subset \cp \times \cp$ is the diagonal.  We also define
\begin{equation} \label{JCAL} {\mathcal{J}_N^{h,\nu}(\mu_{\zeta})=\log
\|e^{U_h^{\mu_{\zeta}}}\|_{L^{kN}(\nu)}}. \end{equation} It is
almost the same functional of the same notation in  \cite{ZZ},
Section 4.7, except that the  $L^N$ norm there now becomes the
$L^{kN}$ norm.

We define the approximate rate functional by  \begin{equation}
\label{IN} - N^2 I_N(\mu_{\zeta}) : =
-\frac{1}{2}\mathcal{E}_N^h(\mu_{\zeta})+\frac{N+1}{N}\mathcal{J}_N^{h,\nu}(\mu_{\zeta}).
\end{equation}

The following is the analogue of Lemma 18 of \cite{ZZ}.

\begin{prop}\label{MAIN1}  With the same notation as in
Proposition \ref{FSVOLZETA2intro}, we have
$$\vec K^N(\zeta_1, \dots, \zeta_N) = \frac{\Gamma_N(\zeta_1,
\dots, \zeta_N)}{\hat{Z}_N(h)} e^{- N^2
\left(-\frac{1}{2}\mathcal{E}_N^h(\mu_{\zeta})+\frac{N+1}{N}\mathcal{J}_N^{h,\nu}(\mu_{\zeta}))\right)}.
$$
\end{prop}

The proof is the same calculation as in \cite{ZZ} and we therefore
omit most of the details. Indeed, the remainder of the proof of
Theorem \ref{POTENTIAL} for $P(\phi)_2$ measure without kinetic
term is identical to that of Theorem 1 of \cite{ZZ}, since the
only change in the approximate rate functional is the change $1
\to k$  in $\mathcal{J}_N^{h,\nu}$ and the factor $\Gamma_N$. The
change in $\jcal_N^{h, \nu}$  cancels out in the limit, since  (as
in \cite{ZZ}), we have
$$\lim _{N\rightarrow \infty}\mathcal{J}_N^{h,\nu}(\mu_{\zeta})=\log
\|e^{U_h^{\mu_{\zeta}}}\|_{L^{kN}}(\nu)\uparrow \log
\|e^{U_h^{\mu_{\zeta}}}\|_{L^{\infty}}(\nu)=\sup_K
U_h^{\mu_{\zeta}}.$$ We briefly re-do the calculation for the sake
of completeness, referring to \cite{ZZ} for further details:
\begin{equation} \begin{array}{lll}\int_{\CP^1} \prod_{j = 1}^N |(z -
\zeta_j)|^{2k} e^{-  k N \phi} d\nu(z) & = & \left( \int_{\CP^1}
e^{k \int_{\CP^1} G_{h}(z,w) dd^c \log ||s_{\zeta}(w)||_{h^N}^2}
d\nu \right) e^{ k \int_{\CP^1} \log
||s_{\zeta}||_{h^N}^2(z) \omega_h} \\ && \\
& = & \left( \int_{\CP^1} e^{k N \int_{\CP^1} G_{h}(z,w)
d\mu_{\zeta}(w)} d\nu \right) e^{ k \int_{\CP^1} \log
||s_{\zeta}||_{h^N}^2(z) \omega_h}.
\end{array} \end{equation}
The right side is then raised to the power $- \frac{N + 1}{k}$. If
we take $\frac{1}{N^2} \log$ of the result we get the supremum of
$\int_{\CP^1} G_{h}(z,w) d\mu_{\zeta}(w) $ on the support of
$d\nu$.

 Further, by
Proposition \ref{FSVOLZETA2intro} the $\Gamma_N$ factor does not
contribute to the rate function $I^{h, K}$. Therefore the special
case of Theorem \ref{POTENTIAL} for $P(\phi)_2$ measures where the
$\|\nabla s\|_{(h^N\otimes g,\nu)}^2$ term is omitted follows from Proposition
\ref{MAIN1} and from the proof of Theorem 1 in \cite{ZZ}.

\end{proof}

\section{\label{KINETIC} Large deviations for Lagrangians with  kinetic term.}

We now include the kinetic energy term. In order to define $\nabla
s$ we need to introduce a connection $\nabla: C^{\infty}(\CP^1,
\ocal(1)) \to C^{\infty} (\CP^1, \ocal(1) \otimes T^*)$. To define
the norm-square $||\nabla s||_{(H^N\otimes g,\nu)}^2$ we  introduce a metric $g$ on $\CP^1$
and  a Hermitian metric $H$ on $\ocal(1)$  to define $|\nabla
s|^2_{H^N \otimes g}$ pointwise and a measure $d\mu$ on $\CP^1$ to integrate
the result.
 The kinetic  term is independent of the potential term, and we
 could choose $H, \mu$ differently from $h, \nu$ in the potential
 term. But to avoid excessive technical complications, we choose
 the metrics and connections to be closely related to those in the
 potential term.

We first assume that  $h=e^{-\phi}$ is  a hermitian metric on
$\ocal(1)  \rightarrow \CP^1$ with positive $(1,1)$ curvature,
$\omega_h = \frac{i}{\pi} \ddbar \phi>0$. We then choose $\nabla$
to be the Chern connection of $h$. Thus, $\nabla s \in
C^{\infty}(\CP^1, \ocal(N) \otimes T^{*(1,0)})$ if $s \in
H^0(\CP^1, \ocal(N))$. We fix a local frame $e$ over $\C$ and
express holomorphic sections of $\ocal(N)$ as $s = p e^N$. The
connection 1-form is defined by $\nabla e = e \otimes \alpha$ and
in the case of the Chern connection for $h$ it is given by $\alpha
= h^{-1} \partial h = \partial \phi$. We further fix a smooth
Riemannian metric $g$ on $\CP^1$ (which could be $\omega_h$ but
need not be).

We  assume that the auxiliary probability measure $d\nu$ on
$\CP^1$ satisfies the following  $L^2$-condition: There exists $r \geq 0$
so that \begin{equation}\label{condition} \int_{\CP^1} |p|^2e^{-N\phi}\omega_h \leq C N^r\int_{\CP^1} |p|^2e^{-N\phi}d\nu,\end{equation}
   for all $p \in \pcal_N$. That is, the inner
   product defined by $(h^N, \omega_h)$  is polynomially bounded
   by the inner product defined by $(h^N, \nu)$. We say that $(h,
   \nu)$ is {\it kinetic admissible} if the data satisfies these
   conditions. The metrics $h$ and $g$ and the measure $\nu$ induce
inner products on $\Gamma(L^N \otimes T^{*(1, 0)})$ by
$$\langle s \otimes dz, s \otimes dz \rangle_{h^N \otimes g} = \int_{\CP^1} (s, s)_{h^N}
(dz, dz)_g\; d\nu. $$

Since $\nabla p e^N = e^N \otimes  \partial p + N p e^N \otimes
\alpha, $ the kinetic energy is given in the local coordinate as
\begin{equation}\label{kic}\begin{array}{lll} \int_{\CP^1} |\nabla s|^2_{h^N \otimes g}d\nu : & = & \int_{\C} (e^N \otimes  \partial p +
N p e^N \otimes \alpha, e^N \otimes  \partial p + N p e^N \otimes
\alpha)_{h^N \otimes g} d\nu \\ && \\
& = & \int_{\C} \left( | \partial p |^2_g  + N p (\alpha,
\partial p)_g + N \bar{p} (\partial p, \alpha)_g + N^2 |p|^2
|\alpha|^2_g \right) e^{- N \phi} d\nu
\end{array}\end{equation}

\subsection{Kinetic admissible $(h, \nu,\nabla)$}

We now show that some natural choices of $(h, \nu,\nabla)$ are kinetic
admissible.

We first observe that $\frac{1}{N} \nabla$ is a bounded operator
on $H^0(M, L^N)$ for any positive line bundle $L$ over the projective \kahler manifold $M$, when the inner
product is defined by a smooth volume form. This is an obvious
result of Toeplitz calculus but we provide a proof using the
Boutet de Monvel-Sj\"ostrand parametrix for the \szego kernel. It
is at this point that we need the assumption that $\omega_h > 0$.

\begin{lem} \label{volume} Assume $h = e^{-\phi}$ is a Hermitian metric on a holomorphic line bundle $L \rightarrow M$
over any compact projective \kahler manifold with $\omega_h=\partial \bar \partial \phi > 0$.  Assume
$d\nu$ is a smooth volume form and $g$ is a Riemannian metric over $M$. Then we have
$$\|\nabla s\|_{(h^N \otimes g,\nu)}^2\leq C(h,g,\nu) N^2 \|s\|^2_{(h^N,\nu)} $$
where $s$ is the holomorphic section of line bundle
$L^N$.\end{lem}

\begin{proof} Let  $\Pi_{N,\nu}: L^2(M,L^N)\rightarrow H^0(M,L^N)$ be the orthogonal projection
with respect to the inner product,$$\langle f_1\otimes e^N,f_2\otimes e^N \rangle=\int_M f_1\bar f_2 e^{-N\phi}d\nu,$$ Here in the local coordinate, we write
$s=fe^N$ as the section of the line bundle $L^N$.

 Let $ \Pi_{N,\nu}(z,w)$ be
its Schwartz kernel with respect to $d\nu$,
$$(\Pi_{N,\nu} s)(z)=\int_M \Pi_{N,\nu}(z,w)f(w) e^{-N\phi(w)}d\nu(w)$$
Then  Bergman kernel has the paramatrix \cite{BBS, BS}
$$\Pi_{N,\nu}(z,w)=e^{N\phi(z\cdot w)}A_{N}e^N(z)\otimes \bar
e^N(w)$$ where $A_N$ is a symbol of order $m = \dim M$ depending
on $h$ and $\nu$ and where $\phi(z\cdot w)$ is the almost-analytic
extension of $\phi(z)$. It follows that the Schwartz kernel of
$\frac{1}{N} \nabla \Pi_{N, \nu}$ has the local form,
$$\begin{array}{l} \frac{1}{N}\nabla\Pi_{N,\nu}(z,w) = \left((\frac{1}{N} \partial +\partial\phi
dz) e^{N\phi(z\cdot w)}A_{N}(z,w) \right)e^N(z) \otimes \bar{e}^N(w)\\
\\
=\left((\partial \phi+\partial_z \phi(z\cdot w) + \frac{1}{N}
\partial \log A_N) e^{N\phi(z\cdot w)}A_{N}\right) e^N(z) \otimes \bar{e}^N(w).
\end{array}$$ Put $\Phi(z,w):=\partial \phi+\partial_z \phi(z\cdot w)
+\partial \log A_N$. Denote by $\Phi \Pi_{N, \nu}$ the product of
$\Phi$ and the Schwartz kernel of $\Pi_{N, \nu}$. Then,
$$\|\frac{1}{N}\nabla s\|_{(h^N\otimes g,\nu)}^2=
\|\frac{1}{N}\nabla\Pi_{N,\nu} s\|_{(h^N\otimes g,\nu)}^2=
\|(\Phi\Pi_{N,\nu}) s\|_{(h^N\otimes g,\nu)}^2$$

We now claim that
$$\|(\Phi\Pi_N)s\|_{L^2(h^N\otimes g,\nu)}\leq C \|s\|_{L^2(h^N,\nu)}. $$
This follows from the  Schur-Young bound on the $L^2 \to L^2$
mapping norm of the integral operator $\Phi\Pi_N$,
\begin{equation} \label{SY} \|\Phi\Pi_N\| \leq C \sup_M \int_M
|\Pi_{N,\nu}(z,w)|d\nu(z), \end{equation} since for any metric $g$ on
$M$, $|\Phi|_{g}\leq C$ uniformly on $M$. To estimate the norm, we
use the following known estimates on the Bergman kernel (see
\cite{SZ4} for a similar estimate and for background): when
$d(z,w)\leq CN^{-\frac{1}{3}}$, we have
$$|\Pi_{N,\nu}(z,w)|_{h^N \otimes h^N}\leq C N^m
e^{-\frac{1}{4}Nd^2(z,w)}+O(N^{-\infty}),$$ and  in general,
$$|\Pi_{N,\nu}(z,w)|\leq C N^m e^{-\lambda \sqrt{N}d{(z,w)}}$$ for
some constant $C$ and $\lambda$.

Since we assume $d\nu$ is a volume form on $M$,  there exists a
positive function $J\in C^\infty(M)$ such that $d\nu=J
\omega_h^m$. We break up the right side of (\ref{SY}) into
$$ \int_{d(z,w)\leq
N^{-1/3}}+\int_{d(z,w)\geq N^{-1/3}}.$$ The first term is bounded
by
$$\begin{array}{l} \leq CN^m \int_{d(z,w)\leq N^{-1/3}} e^{-\frac{1}{4}Nd^2(z,w)} J
\omega^m(z)  \\ \\ \leq  C(h,\nu)N^m \int_{0}^\infty
e^{-\frac{1}{4}N\rho^2}d \rho^{2m} +O(N^{-\infty})\leq C'(h,\nu)\end{array}$$ The
second term is bounded by
$$\begin{array}{l} \leq CN^m \int_ {d(z,w)\geq N^{-1/3}}e^{-\lambda \sqrt{N}d(z,w)}J\omega^m \\ \\ \leq CN^m
\int_{M}e^{-\lambda N^{\frac{1}{6}}} d\nu\leq O(N^{-\infty}) \end{array}$$ as $N$ large enough. Thus the operator norm $\Phi\Pi_N$ is bounded by some constant $C'(h,g,\nu)$.
\end{proof}

\begin{rem} The assumption that $d\nu$ is a smooth volume form
allows us to take the adjoint of $\nabla$. \end{rem}

We now give a more general estimate.
 We assume again that  $h=e^{-\phi}$ has positive curvature $\omega_h > 0$.
  But we now relax the assumption that $d\nu$ is a smooth volume form, and only assume that  $d\nu$ satisfies the $L^2$
   condition:$$ \int_M |s|^2e^{-N\phi}\omega_h^m \leq C N^r\int_M |s|^2e^{-N\phi}d\nu$$  for any
   $s \in H^0(M, L^N)$ and for some $r\geq 0$.

    \begin{lem} \label{carl} Let $\dim M = m$. Under the above assumptions,
we have $$\|\nabla s\|_{(h^N\otimes g,\nu)}^2\leq C N^{r+2m + 2}
\|s\|^2_{(h^N,\nu)}
$$ where $s \in H^0(M, L^N)$. \end{lem}
\begin{proof}First we consider the following Bergman kernel $\Pi_{N,\omega_h}(z,w)$  with respect to the inner
product, $$\Pi_N (f e^N)(z)=\int _M \Pi_{N,\omega_h}(z,w) f
(w)e^{-N\phi} \omega_h^m(w) .$$ As above, we write
$\Phi(z,w)=\partial \phi+\partial_z \phi(z\cdot w) +\partial \log
A_N$.

By  Schwartz' inequality, we have (in an obvious notation)
$$\begin{array}{l}\|\frac{1}{N}\nabla \Pi_{N,\omega_h} s\|^2_{L^2(h^N, \nu)}\\ \\ \leq (\int_M |f|^2 e^{-N\phi}
\omega_h^m)(\int_{M\times M} |\Phi|^2|\Pi_{N,\omega_h}|^2e^{-N\phi(w)-N\phi(z)}
|dz|_g^2 \omega^m_h(w) d\nu(z))\end{array}$$ Since $|\Phi\Pi_{N,\omega_h}|^2 |dz|_g^2 \leq
CN^{2m}$ uniformly, this implies
$$\begin{array}{lll}\|\nabla s\|^2_{L^2(h^N\otimes g, \nu)} &= &\|\nabla \Pi_{N,\omega_h} s\|^2_{L^2(h^N\otimes g, \nu)}  \\ && \\
& \leq & C N^{2m + 2}\int_M |f|^2
e^{-N\phi} \omega_h \leq CN^{r+2 m + 2} \|s\|^2_{L^2(h^N, \nu)},\end{array} $$
under the $L^2$ condition.
\end{proof}

\subsection{Proof of  Theorem \ref{POTENTIALKINETIC}.}

We now prove Theorem \ref{POTENTIALKINETIC}. At first one might
expect the kinetic term to dominate the action, since its square
root  is the $H^1_2(d\nu)$ norm of $s$ and since that norm cannot
be bounded by the $L^p$ norm for any $p < \infty$, at least when
$\nu$ is a smooth area form.  However, we are only integrating
over holomorphic sections of $\ocal(N)$ and with the admissibility
assumption, the ratios of all norms are bounded above and below by
positive constants depending on $N$. Taking logarithm asymptotics
erases any essential difference between these norms.

The main step in the proof is the following generalization of
Proposition \ref{FSVOLZETA2intro}.

\begin{prop} \label{FSVOLZETA2introb} Let $(\pcal_N, \gamma_N)$
be the $P(\phi)_2$ ensemble with action (\ref{S}),  where $(h,
\nabla, \nu)$ is kinetic admissible.  Let $\vec K^N$ be the joint
probability current (\ref{JPCDEF}). Then,
\begin{eqnarray}
    \label{eq-030209b}
    \vec K^N(\zeta_1, \dots, \zeta_N) & = &  \frac{(\tilde{\Gamma}_N(\zeta_1, \dots
\zeta_N))}{\hat{Z}_N(h)} \frac{\exp \left( \sum_{i < j}
G_{h}(\zeta_i, \zeta_j) \right) \prod_{j = 1}^N e^{- 2 N
\phi(\zeta_j)} d^2 \zeta_j  }{\left(\int_{\CP^1} e^{k N
\int_{\CP^1} G_{h}(z,w) d\mu_{\zeta}} d\nu(z)
\right)^{\frac{N+1}{k}}}.
\end{eqnarray}
where $$(**) \;\;\sup_{\{\zeta_1, \dots, \zeta_N\} \in
(\CP^1)^{(N)}} \frac{1}{N^2} \log \tilde{\Gamma}_N (\zeta_1,
\dots, \zeta_N) \to 0$$ and where $ Z_N(h)$, resp. $\hat{Z}_N(h)$,
is the normalizing constant in Proposition 3 of \cite{ZZ}.
\end{prop}

 \begin{proof}

 We closely follow the proof of Proposition \ref{FSVOLZETA2intro},
 and do not repeat the common steps.
For the $P(\phi)_2$ measures (\ref{GL}) with kinetic term,
\begin{equation} \label{JPCc} \begin{array}{lll}  \dcal(a_0; \zeta_1, \dots, \zeta_N)&
=& e^{- \int_{\CP^1} (|\nabla \prod_{j = 1}^N (z - \zeta_j)|^2 + P
(|a_0|^2 | \prod_{j = 1}^N (z - \zeta_j)|^2_{h^N}) d\nu(z)} \\
&& \\ &= & e^{-(\alpha_k|a_0|^{2k}+\alpha_{k-1}c_{k - 1}
|a_0|^{2k-2}+\cdots + \alpha_1 c_1  |a_0|^2 + \eta |a_0^2|)},
\end{array} \end{equation}
where
\begin{equation} \label{eta} \eta = |a_0|^{-2}\|\nabla s\|^2_{L^2(h^N\otimes g,\nu)}. \end{equation}

Thus, the addition of the kinetic term changes the pushed forward
probability density from (\ref{DCAL})  to
$$\begin{array}{l} \int \dcal(a_0; \zeta_1,
\dots, \zeta_N) |a_0|^{2 N} d^2a_0 \\ \\  =
\int_{\mathbb{C}}e^{-(\alpha_k|a_0|^{2k}+\alpha_{k-1}c_{k - 1}
|a_0|^{2k-2}+\cdots + c_1 \alpha_1 |a_0|^2 +  \eta
|a_0|^2)}|a_0|^{2N}da_0 \wedge d \bar a_0 \\ \\
= \int_{0}^\infty e^{-(\alpha_k\rho^{k}+\alpha_{k-1}c_{k - 1}
\rho^{k-1}+\cdots + c_1 \alpha_1 \rho +  \eta \rho)}\rho^{N}d\rho,
\end{array}$$ where $\rho=|a_0|^2$ and
   $\alpha_i$ is defined by (\ref{alphaj}). We only need to
   understand the effect of the new $\eta$ term.

 We  change variable $\rho\rightarrow \rho \alpha _k^{\frac{1}{k}}$, to get  $$
 \int \dcal(a_0; \zeta_1,
\dots, \zeta_N) |a_0|^{2 N} d^2a_0  = \alpha _k^{\frac{N+1}{k}}
\tilde{\Gamma}_N(\zeta_1, \dots, \zeta_N), $$ where
$$\tilde{\Gamma}_N(\zeta_1, \dots, \zeta_N) : = \int_{0}^\infty e^{-(\rho^{k}+\beta_{k-1}c_{k - 1}
\rho^{k-1}+\cdots + c_1 \beta_1 \rho + \frac{
\eta}{\alpha_k^{\frac{1}{k}}} \rho)}\rho^{N}d\rho.$$ This is the
same expression as in Proposition \ref{FSVOLZETA2intro} except
that the $\Gamma_N$ factor has changed. Hence to prove (**),  it
suffices to prove
$$\frac{1}{N^2}\log \int_{0}^\infty e^{-(\rho^{k}+\beta_{k-1}c_{k
- 1} \rho^{k-1}+\cdots + c_1 \beta_1 \rho + \frac{
\eta}{\alpha_k^{\frac{1}{k}}} \rho)}\rho^{N}d\rho \rightarrow 0.$$

We first prove that the limit is bounded above by $0$. Since the
addition of the positive quantity $\eta \alpha_k^{-\frac{1}{k}}$
 increases the exponent, we have
 $$\begin{array}{l}\frac{1}{N^2}\log \int_{0}^\infty
e^{-(\rho^{k}+\beta_{k-1}c_{k - 1} \rho^{k-1}+\cdots + c_1 \beta_1
\rho + \frac{ \eta}{\alpha_k^{\frac{1}{k}}} \rho)}\rho^{N}d\rho \\
\\ \leq \frac{1}{N^2}\log \int_{0}^\infty
e^{-(\rho^{k}+\beta_{k-1}c_{k - 1} \rho^{k-1}+\cdots + c_1 \beta_1
\rho )}\rho^{N}d\rho,  \end{array}$$ so the integral is bounded
above by its analogue in the pure potential case, and  it follows
from the proof in section \ref{without} that the last integral
tends to 0.

 We now consider the lower bound. By Lemmas  \ref{volume} and \ref{carl} (with $m = 1$) and by
H\"{o}lder inequality,  we have
 $$\eta\leq CN^n|a_0|^{-2}\|s\|^2_{L^2(h^N,\nu)}\leq CN^n \alpha_k^{\frac{1}{k}},$$
  in the cases  $n=2$ with  $\nu$  a smooth volume form or $n\geq 4$ when $\nu$ satisfies the weighted
  $L^2$ Bernstein inequality (\ref{BERN}). We then have,

   $$\begin{array}{l} \frac{1}{N^2}\log \int_{0}^\infty
e^{-(\rho^{k}+\beta_{k-1}c_{k - 1} \rho^{k-1}+\cdots + c_1 \beta_1
\rho + \frac{ \eta}{\alpha_k^{\frac{1}{k}}} \rho)}\rho^{N}d\rho\\
\\
\geq \frac{1}{N^2}\log \int_{0}^\infty
e^{-(\rho^{k}+\beta_{k-1}c_{k - 1} \rho^{k-1}+\cdots + c_1 \beta_1
\rho + CN^n \rho)}\rho^{N}d\rho \\ \\
\geq \frac{1}{N^2}\log \int_{0}^\infty
e^{-(\rho^{k}+\beta_{k-1}|c_{k - 1}| \rho^{k-1}+\cdots + |c_1|
\beta_1 \rho + CN^n \rho)}\rho^{N}d\rho.
\end{array}$$  Hence,  it
suffices to prove that $$\frac{1}{N^2}\log \int_{0}^\infty
e^{-(\rho^{k}+\beta_{k-1}|c_{k - 1}| \rho^{k-1}+\cdots + |c_1|
\beta_1 \rho + CN^n \rho)}\rho^{N}d\rho \geq 0. $$

We use the steepest descent  method to show that the latter tends
to zero. The maximum of the phase function occurs when
$$k\rho_N^{k}+(k-1)\beta_{k-1}|c_{k - 1}|
\rho_N^{k-1}+\cdots + |c_1| \beta_1 \rho_N + CN^n \rho_N =N.$$ It
follows  first that  $\rho_N \leq \frac{1}{CN^{ n-1}}<1$.  Thus
$$\begin{array}{lll}N&=&k\rho_N^{k}+(k-1)\beta_{k-1}|c_{k - 1}| \rho_N^{k-1}+\cdots +
|c_1| \beta_1 \rho_N + CN^n \rho_N\\ &&\\ &\leq & k\rho_N+(k-1)\beta_{k-1}|c_{k - 1}|
\rho_N+\cdots + |c_1| \beta_1 \rho_N + CN^n \rho_N
\end{array}$$which implies $$\rho_N \geq
\frac{N}{C(k,c_{k-1},\cdots,c_1)+CN^n},$$ and therefore $$\rho_N
\sim \frac{1}{C'N^{n-1}}$$ for $N$ large enough. Thus by the
formula of steepest descent, $$\begin{array}{l}\frac{1}{N^2}\log \int_{0}^\infty
e^{-(\rho^{k}+\beta_{k-1}|c_{k - 1}| \rho^{k-1}+\cdots + |c_1|
\beta_1 \rho + CN^n \rho)}\rho^{N}d\rho \\ \\ \sim \frac{1}{N}\log\rho_N-\frac{1}{N^2}( \rho_N^{k}+\beta_{k-1}|c_{k
- 1}| \rho^{k-1}_N+\cdots + |c_1| \beta_1 \rho_N + CN^n
\rho_N)\\ \\ \sim -\frac{(n-1)\log
(C'N)}{N}-\frac{1}{N^2}((\frac{1}{C'N^{n-1}})^k+\cdots+\beta_1|c_1|\frac{1}{C'N^{n-1}})-C\frac{1}{C'N}
\end{array}$$ which goes to $0$ as $N\rightarrow \infty$, and (**)
holds.
\end{proof}

This completes the proof of Proposition \ref{FSVOLZETA2introb}.
The rest of the proof proceeds exactly as in \S \ref{COMPLETE},
completing the proof of Theorem \ref{POTENTIALKINETIC}.

\end{document}